\documentclass[12pt,nopagetitle]{article}
\usepackage{amssymb,amsmath, amsthm}
\usepackage[english]{babel}
\usepackage{textcomp}
\begin{document}

\theoremstyle{plain}
\newtheorem{theo}{Theorem}
\newtheorem{pred}{Statement}
\newtheorem{lem}{Lemma}

\begin{center}
{\large \bf On the irrationality measure of certain numbers}

\medskip
{\bf A.~Polyanskii\footnote[1]{Moscow Institute of Physics and Technology
E-mail: alexander.polyanskii@yandex.ru. 
The research was partially supported by the Russian Foundation for Basic Research, 
grant \textnumero 09-01-00743.}}

\abstract{
The paper presents upper estimates for the irrationality measure and the non-quadraticity measure for the numbers 
$\alpha_k=\sqrt{2k+1}\ln\frac{\sqrt{2k+1}-1}{\sqrt{2k+1}+1}, \
k\in\mathbb N.$}
\end{center}

For any irrational $\alpha$, \emph{the irrationality measure} can be defined  
as the exact upper bound on the numbers $\kappa$ such that the inequality
$$\left|\alpha -\frac{p}{q}\right|<q^{-\kappa}$$
has infinitely many rational solutions  $\frac{p}{q}$.
By $\mu(\alpha)$ denote the irrationality measure of $\alpha$.

\emph{The non-quadraticity measure} can be defined for any real  $\alpha$ which isn't 
a root of a quadratic equation as the exact upper bound on the numbers $\kappa$ such that the inequality
$$\left|\alpha - \beta\right|< H^{-\kappa}(\beta)$$
has infinitely many rational solutions in quadratic irrationalities $\beta$. 
Here $H(\beta)$ is the height of the characteristic polynomial of $\beta$ 
(taking an irreducible integer polynomial with one of the roots equal to $\beta$, 
$H$ the largest absolute value of this polynomial's coefficients). 
By $\mu_2(\alpha)$ denote the non-quadraticity measure of $\alpha$.

We are going to preset improved bounds on the irrationality measure and 
the non-quadraticity measure for the numbers
\begin{equation}\label{alpha}\alpha_k =\sqrt{2k+1}\ln \frac{\sqrt{2k+1}-1}{\sqrt{2k+1}+1},
\text{ where } k\in \mathbb{N}.\end{equation}

This paper is, in a sense, a continuation of the paper \cite{Bash2} by M. Bashmakova: 
the same integral is considered, but the denominator is estimated more accurately 
by using a certain coefficient symmetry. 
Some earlier estimates for  the irrationality measure of the numbers investigated
by the author have been obtained by A.~Heimonen, T.~Matala-Aho, and K.~V\"a\"an\"anen \cite{Heimonen}, M.~Hata \cite{Hata2},
G.~Rhin \cite{Rhin}, E.~Salnikova \cite{Saln}, M.~Bashmakova \cite{Bash1}, \cite{Bash2}.

Some of the numerical results obtained in this paper have been summarized in the table below:
\begin{center} \begin{tabular}{|r|r|r|r|r|r|r|r|r|r|}
 \hline
 $k$ &  $\mu \left(\alpha_k\right)\leqslant$ & $\mu_2 \left(\alpha_k\right)\leqslant$ & $k$ & 
 $\mu \left(\alpha_k\right)\leqslant$& $\mu_2 \left(\alpha_k\right)\leqslant$ &$k$ & 
 $\mu \left(\alpha_k\right)\leqslant$& $\mu_2 \left(\alpha_k\right)\leqslant$ \\
\hline
\hline $3$ & $6.64610\dots$&---          & 7 & $5.45248\dots$ &---       & $10$ & $3.45356\dots$ &$10.0339\dots$\\
\hline $5$ & $5.82337\dots$&---          & 8 & $3.47834\dots$ &$10.9056\dots$ & $11$ & $5.08120\dots$ &---\\
\hline $6$ & $3.51433\dots$&$12.4084\dots$ & 9 & $5.23162\dots$ &---       & $12$ & $3.43506\dots$ &$9.46081\dots$\\
\hline
\end{tabular}
\end{center}

Let  $n$ be an odd positive integer, and let 
 $a,b$ be fixed positive integers (where $b$ is odd) such that $b>4a$. Define a polynomial as follows:
$$A(x)={x+(b-2a)n\choose (b-4a)n}{x+(b-a)n\choose (b-2a)n} {x+bn\choose
bn}=$$
$$=\frac{(x+2an+1)\dots(x+(b-2a)n)}{((b-4a)n)!}\cdot\frac{(x+an+1)\dots(x+(b-a)n)}{((b-2a)n)!}\cdot\frac{(x+1)\dots(x+bn)}{(bn)!}.
$$
Consider an integral
\begin{equation*} I(z)=\frac{z^{-\frac{bn+1}{2}}}{2\pi i}\int_L A(\zeta)\left(\frac{\pi}{\sin \pi\zeta}\right)^3 (-z)^{-\zeta} d\zeta ,\end{equation*}
where $z\neq0$, the vertical line $L$ is given by the equation
$\mathfrak{Re}\;\zeta= C$, where $-(b-2a)n<C<-2an-1$, and this line is traversed from bottom to top.
We also suppose that  $(-z)^{-\zeta}=e^{-\zeta\ln (-z)}$,
where the branch of the logarithm $\ln (-z)=\ln |z| + i\arg z +i\pi$
is chosen so that $|\arg z|<\pi$.

\begin{pred}
For all  $z\in \mathbb{C}$ such that $0<|z|<1$ we have
 \begin{equation}\label{integ}
 I(z)=-\frac{1}{2}U(z)\ln ^2 z +V(z)\ln z-\frac{1}{2}W(z)-
 i\pi(U(z)\ln z -V(z)),\end{equation}
 where the functions  $U(z),V(z), W(z)\in \mathbb{Q}(z)$ are defined for $|z|<1$ by the following equations:
$$U(z)=-z^{-\frac{bn+1}{2}}\sum_{k=bn+1}^{\infty}A(-k)z^k,$$
 \begin{equation}\label{v_z}V(z)=-z^{-\frac{bn+1}{2}}\sum_{k=(b-a)n+1}^{\infty}A^{\prime}(-k)z^k,\end{equation}
 $$W(z)=-z^{-\frac{bn+1}{2}}\sum_{k=(b-2a)n+1}^{\infty}A^{\prime\prime}(-k)z^k.$$
\end{pred}

The proof of this statement (in a somewhat different form) has been given by Yuri Nesterenko \cite{Nest2}.
He has also proved the following lemma (see Lemma 1 in \cite{Nest2}).

\begin{lem}
Let  $P(x)\in \mathbb{C}[x]$ be a polynomial of degree $d$. Then for all
 $z$, $|z|<1$, we have
$$-\sum_{k=1}^{\infty}P(-k)z^k=\sum_{j=0}^{d}c_j\left(\frac{z}{z-1}\right)^{j+1},$$
where
$$c_j=\sum_{k=1}^{k=j+1}(-1)^{k-1}P(-k){j\choose
k-1}.$$
\end{lem}

It has also been shown (see Statement 2 and Lemma 1 in \cite{Bash2}) that 
\begin{equation}\label{uvw}U(z)=U\left(\frac{1}{z}\right)=
\hat{U}\left(z+\frac{1}{z}\right), V(z)=-V\left(\frac{1}{z}\right)=
\left(z-\frac{1}{z}\right)\hat{V}\left(z+\frac{1}{z}\right),\atop W(z)
=W\left(\frac{1}{z}\right)=\hat{W}\left(z+\frac{1}{z}\right),
\text{ where } \hat{U}(z),\hat{V}(z),\hat{W}(z)\in \mathbb{Q}(z).\end{equation}
If the number $x$ satisfies $x+\frac{1}{x}\in \mathbb{Q}$, then we obtain $U(x),
\frac{V(x)}{x-1/x},W(x)\in \mathbb{Q}$.

Now consider the values of the integral at the following points:
\begin{equation}\label{x} x_k=\frac{k+1-\sqrt{2k+1}}{k}=\frac{\sqrt{2k+1}-1}{\sqrt{2k+1}+1},
\text{ where }k\in \mathbb{N}.\end{equation}
In this case we have $x_k+\frac{1}{x_k}=\frac{2k+2}{k}$ and
$x_k-\frac{1}{x_k}=-2\frac{\sqrt{2k+1}}{k}$. This easily leads to
\begin{equation}\label{uvw2}U(x_k), \sqrt{2k+1}V(x_k),W(x_k)\in \mathbb{Q}.\end{equation}
Let us define $\Omega$ as the set of $0\leqslant y<1$ such that for all $x\in \mathbb{R}$ the inequality
\begin{equation}\label{omega}\left([x-2ay]-[x-(b-2a)y]-[(b-4a)y]\right)+\atop
+\left([x-ay]-[x-(b-a)y]-[(b-2a)y]\right)+\left([x]-[x-by]-[by]\right)\geqslant1\end{equation}
is satisfied.

The set $\Omega$ is a union of points, as well as closed, half-open and open intervals. Clearly, the points of the interval 
$\left[0;\frac{1}{b}\right)$ do not belong to $\Omega$, and thus the set of primes $p>\sqrt{bn}$ such that
$\left\{\frac{n}{p}\right\}\in\Omega$ is finite.
Denote the product of these primes as $\Delta$, and denote as
$\Delta_1$ the product of the primes $p>(b-2a)n$ satisfying
$\left\{\frac{n}{p}\right\}\in\Omega$.

Let $d_n$ be the least common multiple of 
$1,2,\dots,n$.

Let us introduce the following rational numbers:
\begin{equation*}R_{k,n}=\left\{\begin{aligned}& m^{-\frac{bn+1}{2}},& \text{ if } k=2m;\\
& 2^{\frac{3(b-2a)n+1}{2}} k^{-\frac{bn+1}{2}},& \text{ if } k=2m-1. \end{aligned}\right.\end{equation*}
\begin{equation*}S_{k,n}=\left\{\begin{aligned}& m^{-\frac{(b-2a)n+1}{2}},& \text{ if } k=2m;\\
& 2^{\frac{3(b-2a)n+1}{2}} k^{-\frac{(b-2a)n+1}{2}},& \text{ if } k=2m-1. \end{aligned}\right.\end{equation*}
\begin{equation*}T_{k,n}=\left\{\begin{aligned}& m^{-\frac{(b-4a)n+1}{2}},& \text{ if } k=2m;\\
& 2^{\frac{3(b-2a)n+1}{2}} k^{-\frac{(b-4a)n+1}{2}},& \text{ if } k=2m-1. \end{aligned}\right.\end{equation*}

\begin{lem} Let $x_k$ be defined by (\ref{x}), then we have
 $$A=R_{k,n} U(x_k)\in \mathbb{Z},\;
 B=S_{k,n} \frac{d_{bn}}{\Delta}  V(x_k)\sqrt{2k+1}\in
\mathbb{Z},\;
C=T_{k,n} d_{(b-2a)n}\Delta_1\frac{d_{bn}}{\Delta}W(x_k)\in\mathbb{Z}.$$
\end{lem}
\begin{proof} Let us prove the statement for $B$. 
The statements for $A$ and $C$ can be proved by the same argument.

It suffices to show that  $B^2\in \mathbb{K}$, where $\mathbb{K}$ is the ring of 
algebraic integers, which would imply that $B$ is also an algebraic integer. 
Then from the property (\ref{uvw2}) we could say that $B\in\mathbb{Q}$, 
and consequently $B\in\mathbb{Z}$.

Let $A_1(z)$ be a polynomial of degree $3(b-2a)n$ such that 
$A_1(z)=A(z-an)$. Then it's easy to see that 
\begin{equation}\label{nuli}A_1^{\prime}(-1)=\dots=A_1^{\prime}(-(b-2a)n)=A^{\prime}(-an-1)=\dots=A^{\prime}(-(b-a)n)=0.\end{equation}
Let us rewrite (\ref{v_z}). Defining $l$ as $l=k-an$ and applying $(\ref{nuli})$, we obtain
\begin{multline*}V(z)=-z^{-\frac{bn+1}{2}}\sum_{k=(b-a)n+1}^{+\infty}A^{\prime}(-k)z^k=
-z^{-\frac{bn+1}{2}}\sum_{k=an+1}^{+\infty}A^{\prime}(-k)z^k=\\=
-z^{-\frac{bn+1}{2}+a n}\sum_{k=an+1}^{+\infty}A^{\prime}(-k)z^{k-an}=-z^{-\frac{bn+1}{2}+an}\sum_{l=1}^{+\infty}A^{\prime}(-l-an)z^l=
-z^{-\frac{bn+1}{2}+an}\sum_{l=1}^{+\infty}A_1^{\prime}(-l)z^l.\end{multline*}
Then by Lemma 1 we have
\begin{multline*}V(z)=z^{-\frac{bn+1}{2}+an}\sum_{j=0}^{3(b-2a)n-1}
b_j
\left(\frac{z}{z-1}\right)^{j+1},\text{ where }b_j=\sum_{k=1}^{j+1}(-1)^{k-1} A_1^{\prime}(-k){j\choose
k-1}.\end{multline*}
Lemma 4 of \cite{Nest2} states that
$$\frac{d_{bn}}{\Delta}A_1^{\prime}(-k)\in\mathbb{Z}\text{, where }1\leq
k\leq3(b-2a)n, \text{ i.e. } \frac{d_{bn}}{\Delta}b_j\in \mathbb{Z}$$
From (\ref{nuli}) it follows that
\begin{multline*}
V(z)=z^{-\frac{bn+1}{2}+an}\sum_{j=(b-2a)n}^{3(b-2a)n-1}
b_j
\left(\frac{z}{z-1}\right)^{j+1}=\\
=z^{-\frac{bn+1}{2}+an}\left(\frac{z}{z-1}\right)^{(b-2a)n+1}\sum_{j=(b-2a)n}^{3(b-2a)n-1}
b_j \left(\frac{z}{z-1}\right)^{j-(b-2a)n}.\end{multline*} 
It is true that $\frac{x_k}{x_k-1}=\frac{1-\sqrt{2k+1}}{2}=t_1$ and
$\frac{1}{1-x_k}=\frac{1+\sqrt{2k+1}}{2}=t_2$ are the solutions of the equation
$t^2-t-\frac{k}{2}=0$. Clearly, they must be algebraic numbers.

Applying (\ref{uvw}) yields that
\begin{equation*}-\left(
V(x_k)\right)^2=V(x_k)V\left(\frac{1}{x_k}\right)=\left(t_1 t_2\right)^{(b-2a)n+1}
\sum_{j=(b-2a)n}^{3(b-2a)n-1} b_j
t_1^{j-(b-2a)n}\sum_{j=(b-2a)n}^{3(b-2a)n-1}
b_j
t_2^{j-(b-2a)n}=\end{equation*}
\begin{equation}\label{v_x}=\left(\frac{k}{2}\right)^{(b-2a)n+1}
\sum_{j=(b-2a)n}^{3(b-2a)n-1} b_j
t_1^{j-(b-2a)n}\sum_{j=(b-2a)n}^{3(b-2a)n-1} b_j t_2^{j-(b-2a)n}.
\end{equation}

If $k$ is even, then for all positive integers $N$ we have $t_i^N\in
\mathbb{K}$ since $t_i\in
\mathbb{K}$.

If $k$ is odd, then for all positive integers $N$ we can write
$2^{\left[\frac{N+1}{2}\right]}t_i^N\in \mathbb{K}$
since $(t_i)^{2N}=\left(\frac{k+1\pm\sqrt{2k+1}}{2}\right)^N$ and
$(t_i)^{2N+1}=\frac{1\pm\sqrt{2k+1}}{2}\left(\frac{k+1\pm\sqrt{2k+1}}{2}\right)^N$.

Let us consider the following two cases:

1. $k=2m.$ Then by (\ref{v_x})  we have
$$B^2=-\left(\sum_{j=(b-2a)n}^{3(b-2a)n-1}\left(\frac{d_{bn}}{\Delta}b_j \right)t_1^{j-(b-2a)n}\right)
\left(\sum_{j=(b-2a)n}^{3(b-2a)n-1}\left(\frac{d_{bn}}{\Delta}b_j\right)
t_2^{j-(b-2a)n}\right)(2k+1)\in\mathbb{K}.$$

2. $k=2m-1.$ Then (\ref{v_x}) yields that
\begin{multline*}B^2=-\left(\sum_{j=(b-2a)n}^{3(b-2a)n-1}\left(\frac{d_{bn}}{\Delta}b_j \right)2^{(b-2a)n}t_1^{j-(b-2a)n}\right)\times \\
\times
\left(\sum_{j=(b-2a)n}^{3(b-2a)n-1}\left(\frac{d_{bn}}{\Delta}b_j\right)
2^{(b-2a)n}t_2^{j-(b-2a)n}\right)(2k+1)\in\mathbb{K}.\end{multline*}
This concludes the proof of the lemma.
\end{proof}
We proceed by formulating several known results, which have been stated in \cite{Bash2} as Lemmas 5 and 6.
\begin{pred}
Let $x\in\mathbb{R}$, $0<x<1.$ If the equation
$$\frac{z(z-a)(z-2a)}{(z-(b-2a))(z-(b-a))(z-b)}=\frac{1}{x}$$
has a unique solution $z_0>b$, then
$$\lim_{n\rightarrow+\infty}\frac{1}{n}\ln |U(x)|=M_1= 
\ln \left(\frac{(z_0- (b-2a))^{b-2a}(z_0-(b-a))^{b-a}(z_0-b)^b}{(z_0-2a)^{2a}(z_0-a)^a(b-4a)^{b-4a}(b-2a)^{b-2a}b^b}\right)-\frac{b}{2}\ln x.$$
\end{pred}
\begin{pred}
Denote
$$M_1=\ln \frac{|z_1+(b-2a)|^{b-2a}|z_1+(b-a)|^{b-a}|z_1+b|^b}{|z_1+2a|^{2a}|z_1+a|^a(b-4a)^{b-4a}(b-2a)^{b-2a}b^b}-\frac{b}{2}\ln x,$$
where $z_1$ is the complex root of the equation 
$$\frac{(z+2a)(z+a)z}{(z+(b-2a))(z+(b-a))(z+b)}=\frac{1}{x}$$
satisfying the condition $\mathfrak{Im}\;z_1>0.$ Then we have
$$\limsup_{n\rightarrow+\infty}\frac{1}{n}\ln |I(x)|\leqslant M_2.$$
\end{pred}
To compute 
\begin{equation}\label{m}N_1=\lim_{n\rightarrow\infty}\frac{1}{n} 
\ln \frac{d_{bn}}{\Delta}\qquad\text{ and } \qquad N_2=
\lim_{n\rightarrow\infty}\frac{1}{n}\ln d_{(b-2a)n}\Delta_1\frac{d_{bn}}{\Delta}, \end{equation}
we are going to use Lemma 6 from \cite{Nest2}.
\begin{lem}\label{sokr}
Let $u,v$ be real numbers such that 
$0<u<v<1$. Then
\begin{displaymath}
\lim\limits_{n\rightarrow\infty}\frac{1}{n}\sum\limits_{u\leq
\left\{\frac{n}{p}\right\}<v}\ln p=\psi(v)-\psi(u),
\end{displaymath}
where $\psi(x)=\frac{\Gamma^{\prime}(x)}{\Gamma(x)}$ is the logarithmic derivative
of the gamma function, and the sum is taken over all primes  $p$ such 
that the fractional part $\left\{\frac{n}{p}\right\}$
lies in the given range.
\end{lem}
Let us formulate a lemma by Hata (see \cite{Hata1}, Lemma 2.1), 
which will allow us to prove the principal theorem of this paper.
\begin{lem} \label{lem}
Let $n\in\mathbb N$, $\alpha \in\mathbb R$, let $\alpha$ be an irrational number, and let
$l_n=q_n\alpha+p_n$, where $q_n,p_n\in\mathbb Z$ and
\begin{equation*}
\lim\limits_{{n\rightarrow\infty}}\frac{1}{n}\ln|q_n|=\sigma,
\quad \limsup \limits_{n\rightarrow\infty}
\frac{1}{n}\ln|l_n|\leq-\tau,\quad \sigma,\tau>0.
\end{equation*}
Then $$\mu(\alpha)\leq 1+\frac{\sigma}{\tau}.$$
\end{lem}
Now we can state the principal theorem, 
which will be proved by applying the Hata's lemma to the sequence
$$E_n=S_{k,n}\frac{d_{bn}}{\Delta}\left(-\frac{\mathfrak{Im}(I(x_k))}{\pi}\sqrt{2k+1}\right)=
S_{k,n} \frac{d_{bn}}{\Delta}U(x_k) \alpha_k -
S_{k,n}\frac{d_{bn}}{\Delta}V(x_k)\sqrt{2k+1}=
P_n\alpha_k+Q_n.$$

By Lemma 2, we have $P_n,Q_n\in\mathbb{Z}$. Clearly, we can also write
\begin{equation*}
K_1=\lim_{n\rightarrow\infty}\frac{1}{n}\ln (S_{k,n})=
\left\{\begin{aligned}&-\frac{b-2a}{2}\ln m,& \text{ if } k=2m;\\
&-\frac{b-2a}{2}\ln k+\frac{3(b-2a)}{2}\ln 2,& \text{ if }
k=2m-1.
\end{aligned}\right.
\end{equation*}
\begin{theo} \label{t1}
Assume that $a,b\in \mathbb{N}$ satisfy $b>4a$, $x$ is defined by
(\ref{x}), the numbers $M_1$ and $M_2$ are defined by Statements 2 and 3, the set $\Omega$ is defined by
(\ref{omega}), and $N_1$ is defined by (\ref{m}).

If $M_2+K_1+N_1<0$, we have
$$\mu\left(\alpha_k\right)\leq1-\frac{\lim\limits_{n\to\infty}P_n}{\limsup\limits_{n\to\infty}E_n}
\leq 1- \frac{M_1+K_1+N_1}{M_2+K_1+N_1}.$$
\end{theo}

We are going to formulate another lemma by Hata (see \cite{Hata3}, Lemma 2.3), which will allow us to prove another theorem.
\begin{lem}
Let $n\in\mathbb N$, and assume that $\alpha \in\mathbb R$ is not a quadratic irrationality (i.e., not a root of a quadratic integer polynomial). Take
$l_n=q_n\alpha+p_n$, $m_n=q_n\alpha^2+r_n$, where
$q_n,p_n,r_n\in\mathbb Z$, and assume
\begin{equation*}
\lim\limits_{{n\rightarrow\infty}}\frac{1}{n}\ln|q_n|=\sigma,
\quad \max\left\{\limsup \limits_{n\rightarrow\infty}
\frac{1}{n}\ln|l_n|,\limsup \limits_{n\rightarrow\infty}
\frac{1}{n}\ln|m_n|\right\}\leq-\tau,\quad \sigma,\tau>0.
\end{equation*}
Then we have 
$$\mu_2(\alpha)\leq 1+\frac{\sigma}{\tau}.$$
\end{lem}
Now we can easily formulate our second theorem by applying lemma 5 to the following sequences:
\begin{multline*}F_n=T_{k,n}d_{(b-2a)n}\Delta_1\frac{d_{bn}}
{\Delta}\left(-\frac{\mathfrak{Im}(I(x_k))}{\pi} \sqrt{2k+1}\right)=\\
=T_{k,n} d_{(b-2a)n}\Delta_1\frac{d_{bn}}{\Delta}U(x_k) \alpha_k
-T_{k,n}d_{(b-2a)n}\Delta_1\frac{d_{bn}}{\Delta}V(x_k)\sqrt{2k+1}
=X_n\alpha_k+Y_n;\end{multline*}
\begin{multline*}G_n=
T_{k,n}d_{(b-2a)n}\Delta_1\frac{d_{bn}}{\Delta}(2k+1)
\left(2\mathfrak{Re}(I(x_k))-2\frac{\mathfrak{Im}(I(x_k))}{\pi}\right)=\\
=T_{k,n}d_{(b-2a)n}\Delta_1\frac{d_{bn}}{\Delta}U(x_k)\alpha_k^2-
T_{k,n}d_{(b-2a)n}\Delta_1\frac{d_{bn}}{\Delta}(2k+1)W(x_k)=
X_n\alpha_k^2+Z_n.
\end{multline*}
By Lemma 2, we have $X_n,\;Y_n,\;Z_n\in\mathbb{Z}$. We can also see that 
\begin{equation*}
K_2=\lim_{n\rightarrow\infty}\frac{1}{n}\ln (T_{k,n})=\left\{\begin{aligned}&-\frac{b-4a}{2}\ln m,& \text{ if } k=2m;\\
&-\frac{b-4a}{2}\ln k+\frac{3(b-2a)}{2}\ln 2,& \text{ if }
k=2m-1.
\end{aligned}\right.
\end{equation*}
\begin{theo}\label{t2}
Assume that $a,b\in \mathbb{N}$ satisfies $b>4a$, $x_k$ is given by (\ref{x}),
the numbers $M_1$ and $M_2$ are defined by Statements 2 and 3, the set $\Omega$ is defined by
(\ref{omega}), and $N_2$ is defined by (\ref{m}).
\\
If $M_2+K_2+N_2<0$, then we have
$$\mu_2\left(\alpha_k\right)\leq
1-\frac{\lim\limits_{n\to\infty}X_n}
{\max\{\limsup\limits_{n\to\infty}F_n,\;\limsup\limits_{n\to\infty}G_n\}}
\leq 1- \frac{M_1+K_2+N_2}{M_2+K_2+N_2}.$$
\end{theo}

{\bf Remark 1.} In conclusion, let us give the parameter values that have been used to obtain the results 
presented in the beginning of this article. For each of the given values of $k$, the irrationality measure 
$\mu(\alpha_k)$ were obtained by taking $a=1$, $b=7$. Non-quadraticity measures $\mu_2(\alpha_k)$ 
have been derived by taking $a=2$, $b=23$ for $k=6$ and $a=1$, $b=13$ for $k=8,10,12$. Note that we 
couldn't estimate the quadratic irrationality  for odd values of $k$ because of the high growth rate 
of the ``denominators'' denoted as $q_n$ in Lemma 5.

{\bf Remark 2.} Using the same approach the author proved a few theorems. 
If one takes 
$y_k=\frac{k-1-i\sqrt{2k-1}}{k}$ in (\ref{integ}) instead of $x_k$ and uses the sketch of 
the proof Theorem~\ref{t1}, one 
can prove 
\begin{theo} Let $\beta_k=\sqrt{2k-1}\arctan\frac{\sqrt{2k-1}}{k-1}$. Then we have  \label{t3}
\begin{center} \begin{tabular}{|r|r|r|r|r|r|r|r|r|}
 \hline
 $k$&  $\mu \left(\beta_k\right)\leqslant$&$\mu_2 \left(\beta_k\right)\leqslant$& $k$&  $\mu \left(\beta_k\right)\leqslant$& $\mu_2 \left(\beta_k\right)\leqslant$ \\
\hline
\hline $2$ &$4.60105\dots$&---           &$8$ & $3.66666\dots$  &$14.37384\dots$\\
\hline $4$ &$3.94704\dots$&$44.87472\dots$ &$10$ & $3.60809\dots$ & $12.28656\dots$\\
\hline $6$& $3.76069\dots$&$19.19130\dots$ &$12$ & $3.56730\dots$ & $11.11119\dots$\\
\hline
 \end{tabular}
\end{center}
\end{theo}
One can read the proof of Theorem \ref{t3} in \cite{Pol1}.

If one takes 
$y_k=\frac{1}{2}-i\frac{\sqrt{3}}{6}$ in (\ref{integ}) instead of $x_k$ and uses the 
sketch of the proof Theorem~\ref{t1} 
(one needs Lemma \ref{lem} in a somewhat different form), one 
can prove 
\begin{theo}\label{t4}
For any 
$\varepsilon>0$ there exists $0<q(\varepsilon)\in\mathbb{Z}$ such that for any
$q\geqslant q(\varepsilon), p_1,p_2\in \mathbb{Z}$ we have
\begin{equation*}\label{max}\max\left(\left|\ln 3-\frac{p_1}{q}\right|,\left|\frac{\pi}{\sqrt{3}}-\frac{p_2}{q}\right|\right)\geqslant q^{-3.86041\dots-\varepsilon}.
\end{equation*}
\end{theo}
One can read the proof of Theorem \ref{t4} in \cite{Pol1}.

If one considers the integral (\ref{integ}), where
$$A(x)={x+10n+1\choose 10n+1} {x+9n+1\choose 8n+1}{x+8n+1\choose 6n+1}{x+7n+1\choose 4n+1},$$ 
takes $x_k=\sqrt{2k+1}\ln\frac{k+1-\sqrt{2k+1}}{k}$, then using the sketch of the proof 
Theorem~\ref{t2} one can prove
\begin{theo} \label{t5} Let $\alpha_k=\sqrt{2k+1}\ln\frac{\sqrt{2k+1}-1}{\sqrt{2k+1}+1}$. Then we have
\begin{center}
\begin{tabular}{|r|r|r|r|r|r|r|r|r|r|}
 \hline
 $k$ &  $\mu_2 \left(\alpha_k\right)\leqslant$ & $k$ &  $\mu_2 \left(\alpha_k\right)\leqslant$&$k$& $\mu_2 \left(\alpha_k\right)\leqslant$ &$k$& $\mu_2 \left(\alpha_k\right)\leqslant$&$k$& $\mu_2 \left(\alpha_k\right)\leqslant$\\
\hline
\hline $2$ & $18.5799\dots$  & $8$ &  $10.3786\dots$ & $13$ & $245.5913\dots$& $16$& $9.03034 \dots$ & $19$ & $55.9694\dots$\\
\hline $4$ & $12.8416\dots$  & $10$ & $9.86485\dots$ & $14$ & $9.23973\dots$ & $17$ & $71.3960\dots$& $20$& $8.7192\dots$\\
\hline $6$ & $11.2038\dots$  & $12$ & $9.50702\dots$ & $15$ & $105.4297\dots$& $18$ & $8.86054\dots$& $21$ & $47.1243\dots$\\
\hline
 \end{tabular}
\end{center}
\end{theo}

One can read the proof of Theorem \ref{t5} in \cite{Pol2}.

The author would like to express his deep gratitude to his scientific advisor Yuri 
Nesterenko and to Nikolai Moshchevitin for their interest in the research and their timely advice.


\begin{thebibliography}{111}
\bibitem{Bash1} M.~Bashmakova,  
\emph{Approximation of values of the Gauss hypergeometric function by rational fractions}, 
Mathematical Notes, 88:5 (2010), pp. 785--797.

\bibitem{Bash2} M.~G.~Bashmakova, \emph{Estimates for the exponent of irrationality
for certain values of hypergeometric functions,} Moscow Journal of
Combinatorics and Number Theory, Vol.~1, Issue~1, 2011, pp.~67--78

\bibitem{Hata1} M.~Hata, \emph{Rational approximations to $\pi$ and some other
numbers,} Acta Arithm., LXIII.4. 1993, pp.~335--347.

\bibitem{Hata2} M. Hata, \emph{Irrationality measures of the values of hypergeometric functions,}
Acta Arithm., 1992, V. LX.4, pp.~335--349.

\bibitem{Hata3} M.~Hata, \emph{$\mathbb{C}^2$-caddle method and Beuker`s integral}
Trans. Amer. Math. Soc., 352(2000), no.~10, pp.~4557--4583.

\bibitem{Heimonen} A.~Heimonen, T.~Matala-Aho, K.~V\"a\"an\"anen, \emph{An
application of Jacobi type polynomials to irrationality measures},
 Bulletin of the Australian mathematical society, 50:2,
1994, pp.~225--243.

\bibitem{Nest2} Yu.~V.~Nesterenko, \emph{On the irrationality exponent of the number  $\ln2$,}
Mathematical Notes, 88:3 (2010), pp.~530--543.

\bibitem{Pol1} A.~A.~Polyanskii, \emph{On the irrationality measure of certain numbers - II,}
Mathematical Notes, submitted

\bibitem{Pol2} A.~A.~Polyanskii, \emph{Quadratic irrationality exponents of certain numbers,}
Moscow Univ. Math. Bull., 68:5 (2013), pp.~237--240

\bibitem{Rhin} G.~Rhin, \emph{Approximants de Pad\'e et mesures
effectives d'irrationalit\'e}, S\'eminaire de Th\'eorie des
nombres, Progr. in math. 71, 1987, pp.~155--164.

\bibitem{Saln} E.~S.~Salnikova, \emph{On irrationality measures of some values of the Gauss function (in Russian)}
Chebyshevski\u{\i} Sbornik, 2007, Vol.~8:2, pp.~88--96.


\end{thebibliography}
\end{document}